\documentclass[12pt,letterpaper]{amsart}
\usepackage{geometry,mathpazo}
\usepackage{amssymb,setspace,xcolor}
\usepackage[pdftex,breaklinks,colorlinks,
citecolor=blue,%pagebackref=true,
urlcolor=blue,
linkcolor=blue]{hyperref}
\newtheorem{theorem}{Theorem}

\newtheorem{lemma}{Lemma}

\theoremstyle{remark}

\newcommand{\C}{\mathbb{C}}
\newcommand{\zb}{\overline{z}}
\newcommand{\D}{\Omega}

\newcommand{\dbar}{\overline{\partial}}

\onehalfspace

\title{Compactness of Hankel operators with continuous symbols on convex
domains}

\author{Mehmet \c{C}el\.ik}
\address[Mehmet \c{C}elik]{Texas A\&M University-Commerce, 
	Department of Mathematics, Commerce, TX 75429, USA}
\email{mehmet.celik@tamuc.edu}

\author{S\"{o}nmez \c{S}ahuto\u{g}lu}
\address[S\"{o}nmez \c{S}ahuto\u{g}lu]{University of Toledo, 
	Department of Mathematics \& Statistics, Toledo, OH 43606, USA}
\email{sonmez.sahutoglu@utoledo.edu}

\author{Emil J. Straube}
\address[Emil J. Straube]{Texas A\&M University, Department of Mathematics, 
	College Station, TX 77843, USA}
\email{straube@math.tamu.edu}

\date{\today}
\begin{document}

\begin{abstract}
Let $\Omega$ be a bounded convex domain in $\mathbb{C}^{n}$, $n\geq 2$, 
$1\leq q\leq (n-1)$, and $\phi\in C(\overline{\Omega})$. If the Hankel operator 
$H^{q-1}_{\phi}$ on $(0,q-1)$--forms with symbol $\phi$ is compact, then $\phi$ 
is holomorphic along $q$--dimensional analytic (actually, affine) varieties in 
the boundary. We also prove a partial converse: if the boundary contains only 
`finitely many' varieties, $1\leq q\leq n$, and $\phi\in C(\overline{\Omega})$ 
is analytic along the ones of dimension $q$ (or higher), then $H^{q-1}_{\phi}$ 
is compact.
\end{abstract}

\maketitle

%%%%%%%%%%%%%%%%%%%%%%%%%%%%%%%%%%%%%%%%%%
\section{Introduction and results}\label{intro}

Hankel operators on Bergman spaces on bounded pseudoconvex domains with 
symbols that are continuous on the closure of the domain are compact when  
the $\overline{\partial}$--Neumann operator on the domain is compact 
(\cite{FuStraube01, HaslingerBook, StraubeBook}). It is natural to ask what 
happens when the $\overline{\partial}$--Neumann operator is not compact. Must 
there necessarily be noncompact Hankel operators (with, say, symbol continuous 
on the closure of the domain)? The answer is known only in cases where 
compactness (or lack thereof) of the $\overline{\partial}$--Neumann operator 
is understood : when the domain is convex, bounded, in $\mathbb{C}^{n}$ 
 (\cite{FuStraube01}, Remark 2), or when it is a smooth bounded pseudoconvex 
Hartogs domain in $\mathbb{C}^{2}$ (\cite{SahutogluZeytuncu17}). In these cases, 
compactness of all Hankel operators with symbols that are smooth on the closure 
implies compactness of the $\overline{\partial}$--Neumann operator. In general, 
the question is open\footnote{For the case $q=0$ and $N_{1}$. The answer is 
known to be affirmative for Hankel operators on $(0,q)$--forms with $1\leq q\leq (n-1)$
(\cite{CelikSahutoglu14}); the relevant $\overline{\partial}$--Neumann operator
is then $N_{q+1}$. Here, pseudoconvexity is important; see 
\cite{CelikSahutoglu12}.}; having an answer would be very interesting. 
Alternatively, one can consider a particular situation where the 
$\overline{\partial}$--Neumann operator is not compact and ask for necessary 
and/or sufficient conditions on a symbol that will imply compactness of 
the associated Hankel operator. Given that two symbols whose difference extends 
continuously to the boundary as zero yield Hankel operators that agree modulo a 
compact operator, one would in particular like to understand how the interaction 
of the symbol with the boundary affects compactness of the associated Hankel 
operator. It is this question that we are interested in in the current paper.

Such a study was initiated in \cite{CuckovicSahutoglu09}. For symbols smooth on 
the closure of a smooth bounded pseudoconvex domain in $\mathbb{C}^{n}$, the 
authors show, under the condition that the rank of the Levi form is at least 
$(n-2)$, that compactness of the Hankel operator requires that the symbol is 
holomorphic along analytic discs in the boundary. When the domain is also 
convex, they can dispense with the condition on the Levi form.\footnote{In both 
these cases, compactness of the $\overline{\partial}$--Neumann operator was known 
to fail when there are discs in the boundary; \cite[Theorem 1]{SahutogluStraube06}, 
\cite[Theorem 1.1]{FuStraube98}.} Moreover, for (smooth) convex domains in 
$\mathbb{C}^{2}$, holomorphy of the symbol along analytic discs in the boundary 
is also sufficient for compactness of the Hankel operator. Further contributions 
are in  \cite{Le10,CuckovicSahutoglu17,ClosSahutoglu18,ClosAccepted}; 
we refer the reader to the introduction in \cite{ClosCelikSahutoglu18} for a 
summary. The latter authors significantly reduce the regularity requirements 
on both the domain  (Lipschitz in $\mathbb{C}^{2}$ or convex in $\mathbb{C}^{n}$) 
and the symbol  (in $C(\overline{\Omega})$) that is required to infer holomorphicity 
of the symbol  along analytic discs in the boundary from compactness of the 
Hankel operator.  Because compactness of a Hankel operator localizes  
(\cite{ClosCelikSahutoglu18}; see also \cite{Sahutoglu12}), the latter result carries 
over to locally convexifiable domains. In \cite{CelikSahutogluStraube20}, the 
authors, among other things, extend the result from \cite{ClosCelikSahutoglu18} 
on convex domains in $\C^n$ to Hankel operators on ($\overline{\partial}$--closed) 
$(0,q)$--forms with $0\leq q\leq (n-1)$, but with symbol assumed in $C^{1}$ of 
the closure. When the (convex) domain is smooth and satisfies so called maximal 
estimates\footnote{This condition is equivalent to a comparable eigenvalues 
condition for the Levi form of the boundary, see the discussion in 
\cite{CelikSahutogluStraube20} and their references.}, holomorphicity of the 
symbol along $(n-1)$--dimensional analytic (equivalently: affine) polydiscs in 
the boundary suffices for compactness of the associated Hankel 
operator.\footnote{As noted in \cite{CelikSahutogluStraube20}, these two results 
combined imply that a convex domain that satisfies maximal estimates for 
$(0,q)$--forms does not have any analytic varieties of dimension $\geq q$ in its 
boundary except ones in top dimension $(n-1)$ (and their subvarieties). It would 
be desirable to have a direct proof for this fact.} Finally we 
mention the recent \cite{ChengJinWang}, where the authors consider Hankel 
operators with form symbols (replacing multiplication with the wedge product) 
and prove many of the results discussed above for this situation.

Our first result (Theorem \ref{ThmNonCompact}) reduces the 
regularity of the symbol in Theorem 1 in \cite{CelikSahutogluStraube20} to 
$C(\overline{\Omega})$. It also corrects a geometric issue in the proof of 
Theorem 2 in \cite{ClosCelikSahutoglu18} (where the result is given for $q=1$ 
only). 

Since necessary and sufficient conditions for compactness of the 
$\overline{\partial}$--Neumann operator are understood on convex domains
(\cite{FuStraube98, FuStraube01, StraubeBook}), it is reasonable to expect that
when the domain is convex, one should likewise obtain simple necessary and
sufficient conditions on the symbol that will guarantee compactness of the
Hankel operator. That is, one expects the converse of Theorem
\ref{ThmNonCompact} to hold: when the symbol is holomorphic along
$q$--dimensional varieties in the boundary, $H^{q-1}_{\phi}$ should be compact.
As mentioned above, this implication is known for $C^{1}$ convex domains in
$\mathbb{C}^{2}$ and symbol in $C^{1}(\overline{\Omega})$ 
(\cite[Theorem 3]{CuckovicSahutoglu09}, 
\cite[Theorem (\v{C}u\c{c}kovi\'{c}--\c{S}ahuto\u{g}lu)]{ClosCelikSahutoglu18}, 
as well as for some classes of Reinhardt domains in $\mathbb{C}^{2}$, 
with symbol only assumed in $C(\overline{\Omega})$ 
(\cite[Theorem 1]{ClosSahutoglu18}, \cite[Theorem 1]{ClosAccepted}). 
Theorem \ref{converse} says that the implication is true for convex domains in 
higher dimensions as well, when there are only `finitely many' varieties in 
the boundary. We expect this additional assumption to be an artifact of our 
current proof. On the other hand, our result appears to be the first instance of a 
converse to Theorem \ref{ThmNonCompact} in dimension greater than two.

Before stating our results precisely, we recall the notation form
\cite{CelikSahutogluStraube20} (which is fairly standard). For a
bounded domain $\Omega\subset\mathbb{C}^{n}$, denote by $K^{2}_{(0,q)}(\Omega)$
the space of square integrable $\overline{\partial}$--closed $(0,q)$--forms, and
by $A^{2}_{(0,q)}(\Omega)$ the subspace of forms with holomorphic coefficients.
$P_{q}: L^{2}_{(0,q)}(\Omega)\rightarrow K^{2}_{(0,q)}(\Omega)$ is the
orthogonal projection, the Bergman projection on $(0,q)$--forms. For a symbol
$\phi\in L^{\infty}(\Omega)$, the associated Hankel operator is
$H^{q}_{\phi}:K^{2}_{(0,q)}(\Omega)\rightarrow L^{2}_{(0,q)}(\Omega)$,
\begin{align*}
 H^{q}_{\phi}f = \phi f - P_{q}(\phi f),
\end{align*}
with $\|H^{q}_{\phi}\|\leq \|\phi\|_{L^{\infty}(\Omega)}$. $H^{q}_{\phi}$ equals
the commutator $[\phi,P_{q}]$ (since $P_{q}f = f$), so that statements about
(compactness of) Hankel operators may also be viewed as statements about
commutators between the Bergman projection and multiplication operators. 
When the symbol $\phi$ is in $C^{1}(\overline{\Omega})$, Kohn's formula, 
$P_{q} = \overline{\partial}^{*}N_{q+1}\overline{\partial}$, implies
\begin{align}\label{Kohn}
 H^{q}_{\phi}f = \overline{\partial}^{*}N_{q+1}(\overline{\partial}\phi\wedge f);
\end{align}
here, $N_{q+1}$ is the $\overline{\partial}$--Neumann operator on $(0,q+1)$--forms.

We can now state our first result.
\begin{theorem}\label{ThmNonCompact} 
Let $\Omega$ be a bounded convex domain in $\C^n$, $n\geq 2$, and 
$\phi\in C(\overline{\Omega})$. Let $\psi:\mathbb{D}^q\to b\D$ be a 
holomorphic embedding for some $q$ with $1\leq q\leq n-1$.
If $H^{q-1}_{\phi}$ is compact on $A^2_{(0,q-1)}(\D)$ (a fortiori if it is
compact on $K^{2}_{(0,q-1)}(\Omega)$), then $\phi\circ \psi$ is holomorphic.
\end{theorem}

Before stating Theorem \ref{converse}, we first recall the basic facts about
analytic varieties in the boundaries of convex domains from 
\cite[Section 2 and Proposition 3.2]{FuStraube98} and 
\cite[Lemma 2]{CuckovicSahutoglu09}. Suppose
$\psi:\mathbb{D}^{q}\rightarrow b\Omega$ is a holomorphic embedding, where
$\Omega\subset\mathbb{C}^{n}$ is convex. Then the convex hull of the image
$\psi(\mathbb{D}^{q})$ is contained in the intersection of a complex hyperplane
$H$ through $P\in \psi(\mathbb{D}^{q})$ with $\overline{\Omega}$ 
(\cite[Section 2]{FuStraube98}). So it suffices to consider affine varieties in the 
boundary of this kind. Among these varieties through a boundary point $P$, 
there is a unique one, denoted by $V_{P}$, that has $P$ as a relative interior 
point and whose dimension $m$ is maximal. Then $0\leq m\leq (n-1)$, and 
the relative closure $\overline{V_{P}}$ is the intersection of an $m$--dimensional 
affine subspace through $P$ (contained in $H$) with $\overline{\Omega}$.

Our second result is as follows.
\begin{theorem}\label{converse}
Let $\D$ be a bounded convex domain in $\C^n, \phi\in C(\overline{\Omega}),$ and 
$1\leq q\leq n$. Assume that the boundary of $\D$ contains at most finitely many 
disjoint varieties $\{\overline{V_{P_{k}}}\}$ of dimension $q$ or higher as above, 
$k=1, \ldots, N$. Furthermore, assume that $\phi\circ\psi$ is holomorphic for every 
embedding $\psi:\mathbb{D}^q\to b\Omega$. Then $H^{q-1}_{\phi}$ is compact 
on $K^2_{(0,q-1)}(\D)$.  
\end{theorem}

Note that the condition on holomorphicity of $\phi\circ\psi$ just says 
that $\phi|_{V_{P_{K}}}$ is holomorphic on all $V_{P_{k}}$ of dimension $q$
or higher. We also point out that this assumption on $\phi$ does not
imply that the tangential component of $\overline{\partial}\phi$ 
(say when $\Omega$ is smooth) vanishes on the $V_{P_{k}}$ of dimension 
$q$ or higher; the components transverse to the varieties need not be zero. 

If $\Omega$ is assumed $C^{1}$, the relative closures of distinct varieties are
automatically disjoint: if $Q\in
\overline{V_{P_{k_{1}}}}\cap\overline{V_{P_{k_{2}}}}$, then both $V_{P{k_{1}}}$
and $V_{P{k_{2}}}$ have to be contained in the complex tangent space to
$b\Omega$ at $Q$, and considering the convex hull as above produces a variety
which contains them both. However, without a regularity assumption on $b\Omega$,
the supporting complex hyperplane is not unique, and two distinct varieties may
share a boundary point (as on the boundary of $\mathbb{D}\times\mathbb{D}$).

When $q=n$, there are no $q$--dimensional varieties in the boundary, and the
theorem says that $H^{n-1}_{\phi}$ is always (when $\phi\in
C(\overline{\Omega})$) compact. But this is clear from \eqref{Kohn}, at least
when $\phi\in C^{1}(\overline{\Omega})$, because $N_{n}$, and hence
$\overline{\partial}^{*}N_{n}$, is compact (\cite[Theorem 1.1]{FuStraube98}).
When $\phi$ is merely in $C(\overline{\Omega})$, it can be approximated
uniformly on $\overline{\Omega}$ by smooth functions; the corresponding
(compact) Hankel operators converge in norm to $H_{\phi}^{n-1}$.

As pointed out in \cite[Remark 1]{ClosCelikSahutoglu18}, Theorem \ref{converse}
fails on general domains, that is, without the assumption of convexity or a
related condition. This failure is related to the subtleties surrounding
compactness in the $\overline{\partial}$--Neumann problem on general domains
(but absent in the case of convex domains). Namely, there are smooth bounded
pseudoconvex complete Hartogs domains in $\mathbb{C}^{2}$ without discs in the
boundary and noncompact Hankel operators on them whose symbols are smooth 
on the closure. These symbols trivially satisfy the assumption on holomorphicity 
along analytic discs. The domains were originally constructed in \cite{M98} 
as examples of smooth bounded pseudoconvex complete Hartogs domains 
without discs in the boundary whose $\overline{\partial}$--Neumann operator 
$N_{1}$ is nevertheless not compact 
(see also \cite[Theorem 10]{FuStraube01}, \cite[Theorem 4.25]{StraubeBook}). 
But on these domains, noncompactness of $N_{1}$ implies that there are 
symbols smooth on the closure so that the associated Hankel operator is not 
compact (\cite[Theorem 1]{SahutogluZeytuncu17}). 

%%%%%%%%%%%%%%%%%%%%%%%%%%%%%%%%%%%
\section{Proofs}\label{proofs}

The following simple lemma formulates the lack of holomorphicity of a function
without relying on differentiability, in contrast to
\cite{CelikSahutogluStraube20}; this ultimately allows the regularity of the
symbol to be lowered from $C^{1}(\overline{\Omega})$ to $C(\overline{\Omega})$
in Theorem \ref{ThmNonCompact}.

\begin{lemma}\label{LemWeakHolo}
Let $\Omega$ be a domain in $\C^n$ and $\phi\in C(\Omega)$ that is not
holomorphic. 
Then there exist $h\in C^{\infty}_0(\Omega)$ and $1\leq j\leq n$ such that 
$ \int_{\Omega}\phi\overline{h_{z_j}} \neq 0.$
\end{lemma}
\begin{proof} 
Assume that  $\int_{\Omega}\phi\overline{h_{z_{j}}} = 0$ for all $j$ and 
$h\in C^{\infty}_0(\Omega)$. This means that $\overline{\partial}\phi=0$ as a
distribution
on $\Omega$. But $\overline{\partial}$ is elliptic in the interior, so that
$\phi$ is an ordinary holomorphic function\footnote{Alternatively,
$\overline{\partial}\phi=0$ implies $\Delta\phi = 0$, so by Weyl's Lemma, $\phi$
is a $C^{\infty}$ function (and therefore holomorphic).}, a
contradiction. 
\end{proof}
\begin{proof}[Proof of Theorem \ref{ThmNonCompact}]
We start the proof of Theorem \ref{ThmNonCompact} as in \cite[Proof of Theorem
1]{CelikSahutogluStraube20}. 
After dilation, translation and rotation if necessary, we assume that  
$(2\mathbb{D})^q\times\{0\}\subset b\D$ and, seeking to derive a contradiction,
that the function 
$\phi(z_1,\ldots,z_q,0,\ldots,0)$ is not holomorphic on $\mathbb{D}^q$. 
Let us denote 
$\D_0=\{(z_{q+1},\ldots,z_n)\in \C^{n-q}:(0,\ldots,0,z_{q+1},\ldots,z_n)\in
\D\}$, 
the slice transversal to $\mathbb{D}^q$ through the origin. Then 
convexity of $\D$  implies that $\mathbb{D}^q\times \beta \D_0\subset \D$ 
for any $0<\beta\leq 1/2$ (\cite[page 636, proof of $(1)\Rightarrow
(2)$]{FuStraube98}; \cite[Proof of Theorem 2]{ClosCelikSahutoglu18}). 
Lemma \ref{LemWeakHolo} implies that there exist  
$h\in C^{\infty}_0(\mathbb{D}^q)$ and $\delta>0$  such that 
\begin{align*}
\left| \int_{\mathbb{D}^q} \phi(z_1,\ldots,z_q,0,\ldots,0) 
\frac{\partial \overline{h(z_1,\ldots,z_q)}}{\partial
\zb_1}dV(z_1,\ldots,z_q)\right| 
\geq 2\delta.
\end{align*}
Furthermore, continuity of $\phi$ implies that there exists $0<\beta<1/2$ such
that 
\begin{align*}
\left| \int_{\mathbb{D}^q} \phi(z_1,\ldots,z_n) 
\frac{\partial \overline{h(z_1,\ldots,z_q)}}{\partial
\zb_1}dV(z_1,\ldots,z_q)\right| 
\geq \delta
\end{align*}
for all $(z_{q+1},\ldots,z_n)\in \beta \D_0$. As in 
\cite[Proof of Theorem 2]{ClosCelikSahutoglu18} we use their Lemma 5 
to produce a bounded  sequence 
$\{f_j\}_{j=1}^{\infty}=\{f_j(z_{n})\}_{j=1}^{\infty}\subset A^2(\D)$
such that $\|f_j\|_{L^2(\beta\D_0)}=1$ and $f_{j}\rightarrow 0$ weakly in
$A^{2}_{(0,0)}(\Omega)$\footnote{Alternatively, one can modify the argument
below slightly and use the construction in 
\cite[Proof of $(1)\Rightarrow (2)$ in Theorem 1.1]{FuStraube98}.}. 
We define $\alpha_j=f_jd\zb_2\wedge \cdots\wedge d\zb_q.$ Then the sequence
$\{\alpha_{j}\}_{j=1}^{\infty}$ also has (the analogues of) these two
properties.

We have 
 \begin{align}
\nonumber \delta^2=& \delta^2\|f_j\|^2_{L^2(\beta \D_0)}\\
\nonumber \leq & \int_{\beta\D_0}\left|\int_{\mathbb{D}^q}\phi(z_1,\ldots,z_n) 
\frac{\partial \overline{h(z_1,\ldots,z_q)}}{\partial
\zb_1}dV(z_1,\ldots,z_q)\right|^2 
f_j(z_n)\overline{f_j(z_n)}dV(z_{q+1},\ldots,z_n) \\
\nonumber =&  \int_{\beta\D_0}\left|\int_{\mathbb{D}^q}\phi(z_1,\ldots,z_n)
f_j(z_n) 
\frac{\partial \overline{h(z_1,\ldots,z_q)}}{\partial \zb_1}dV(z_1,\ldots,z_q)
\right|^2 dV(z_{q+1},\ldots,z_n) \\
\label{Eqn1} =&  \int_{\beta\D_0}\left|\Big( \phi \alpha_j,   
\frac{\partial h}{\partial z_1}d\zb_2\wedge \cdots \wedge
d\zb_q\Big)_{\mathbb{D}^q}
\right|^2 dV(z_{q+1},\ldots,z_n)\\
\nonumber =&  \int_{\beta\D_0}\left|\Big( \phi \alpha_j,   
\dbar^*_{z_1\cdots z_q}(hd\zb_1\wedge \cdots \wedge d\zb_q)\Big)_{\mathbb{D}^q}
\right|^2 dV(z_{q+1},\ldots,z_n).
\end{align}
In the last two lines $(\cdot,\cdot)$ denotes the standard inner product between
forms on the fibers $\mathbb{D}^{q}\times\{(z_{q+1}, \ldots, z_{n})\}$, and
$\dbar^*_{z_1\cdots z_q}$ is the adjoint of $\overline{\partial}_{z_1\cdots
z_q}$, the $\overline{\partial}$ on each fiber (note that $hd\zb_1\wedge \cdots
\wedge d\zb_q \in dom(\dbar^*_{z_1\cdots z_q})$ since $h\in
C^{\infty}_{0}(\mathbb{D}^{q})$).
In the last equality above we used the fact that all the additional terms 
in $\dbar^*_{z_1\cdots z_q}(hd\zb_1\wedge \cdots \wedge d\zb_q)$ involve 
$d\zb_1$ and are thus (even pointwise) orthogonal to $\phi\alpha_{j}$. 

Next, note that 
\begin{align*}
\dbar_{z_1\cdots z_q} (P_{q-1}\phi\alpha_j|_{\mathbb{D}^q}) 
=(\dbar_{z_1\cdots z_n} P_{q-1}\phi\alpha_j)|_{\mathbb{D}^q} =0,
\end{align*}
where $|_{\mathbb{D}^q}$ denotes the restriction (pull back) of
forms to each fiber $\mathbb{D}^q\times \{(z_{q+1}, \ldots, z_{n})\}$.
Therefore, since $\phi\alpha_{j} = \phi\alpha_{j}|_{\mathbb{D}^q}$ on
$\mathbb{D}^{q}\times \beta\Omega_{0}$,
\begin{align*}
\Big( \phi \alpha_j,   
\dbar^*_{z_1\cdots z_q}hd\zb_1\wedge \cdots \wedge d\zb_q\Big)_{\mathbb{D}^q}
=& \Big( (\phi\alpha_j-P_{q-1}(\phi\alpha_j))|_{\mathbb{D}^q}, 
\dbar^*_{z_1\cdots z_q}hd\zb_1\wedge \cdots \wedge d\zb_q\Big)_{\mathbb{D}^q} \\
=&\Big( (H^{q-1}_{\phi}\alpha_j)|_{\mathbb{D}^q}, 
\dbar^*_{z_1\cdots z_q}hd\zb_1\wedge \cdots \wedge
d\zb_q\Big)_{\mathbb{D}^q}
\end{align*} 
and consequently 
\begin{align*}
\left|\Big( \phi \alpha_j,   
\dbar^*_{z_1\cdots z_n}hd\zb_1\wedge \cdots \wedge  d\zb_q 
\Big)_{\mathbb{D}^q}\right| 
\leq 
\left\|(H^{q-1}_{\phi}\alpha_j)|_{\mathbb{D}^q}\right\|_{L^2(\mathbb{D}^q)} 
   \left\| \dbar^*_{z_1\cdots z_q} 
  hd\zb_1\wedge \cdots \wedge d\zb_q \right\|_{L^2(\mathbb{D}^q)}.
\end{align*}
Combining this estimate with \eqref{Eqn1} and observing that
$|_{\mathbb{D}^{q}}$ decreases norms (pointwise on each fiber $\mathbb{D}^{q}$:
the omitted terms containing  $d\overline{z}_{m}$ with $m>q$ are orthogonal to
$(H^{q-1}_{\phi}\alpha_{j})|_{\mathbb{D}^{q}}$) and that
$\mathbb{D}^{q}\times\beta\Omega_{0} \subseteq \Omega$ gives
\begin{align*}
\delta \lesssim  \left\|(H^{q-1}_{\phi}\alpha_j)\right\|_{L^2(\D)} 
\left\| \dbar^*_{z_1\cdots z_q} 
hd\zb_1\wedge \cdots \wedge d\zb_q \right\|_{L^2(\mathbb{D}^q)}.
\end{align*}
Therefore, $\{H^{q-1}_{\phi}\alpha_j\}_{j=1}^{\infty}$ does not converge to 0 in
$L^2_{(0,q-1)}(\D)$ 
as $j\to \infty$ (since the second factor on the right hand side is nonzero). 

On the other hand, the sequence $\{\alpha_j\}_{j=1}^{\infty}$ converges to zero 
weakly 
in $A^2_{(0,q-1)}(\D)$, and therefore in $K^2_{(0,q-1)}(\D)$, as $j\to \infty$, 
and compactness of $H_{\phi}^{q-1}$ then forces convergence of
$\{H^{q-1}_{\phi}\alpha_j\}_{j=1}^{\infty}$ to zero in $L^2_{(0,q-1)}(\D)$. This
contradiction completes the proof of Theorem \ref{ThmNonCompact}.
\end{proof}
The above proof of Theorem \ref{ThmNonCompact} uses ideas from
\cite{FuStraube98, CuckovicSahutoglu09, ClosCelikSahutoglu18,
CelikSahutogluStraube20}; in turn, these ideas can be traced back at least to
\cite{Catlin81, DiederichPflug81}.

In the proof of Theorem \ref{converse}, we will repeatedly use the following
sufficient condition for compactness of an operator $T:X\rightarrow Y$, where
$X$ and $Y$ are Hilbert spaces (see e.g. \cite[Lemma 4.3(ii)]{StraubeBook}): for
all $\varepsilon>0$ there are a Hilbert space $Z_{\varepsilon}$, a linear
compact operator $S_{\varepsilon}:X\rightarrow Z_{\varepsilon}$, and a constant
$C_{\varepsilon}$ such that
\begin{align}\label{compcond}
 \|Tx\|_{Y} \leq \varepsilon\|x\|_{X} +
C_{\varepsilon}\|S_{\varepsilon}x\|_{Z_{\varepsilon}}.
\end{align}

In addition, we need the following sufficient conditions for compactness of a
Hankel operator on $\Omega$ (notation as in the theorem). 

\begin{lemma}\label{preplemma}
Suppose that $\Omega$ is as in Theorem \ref{converse} and
\begin{itemize}
\item[(i)] $\phi\in C^{1}(\overline{\Omega})$ and $\overline{\partial}\phi$ 
vanishes on $\cup_{j=1}^{N}\overline{V_{P_{j}}}$, or 
\item[(ii)] $\phi\in C(\overline{\Omega})$ vanishes on
$\cup_{j=1}^{N}\overline{V_{P_{j}}}$. 
\end{itemize}
Then $H^{q-1}_{\phi}$ is compact on $K^{2}_{(0,q-1)}(\Omega)$.
\end{lemma}

Note that the condition in $(i)$ is stronger than saying that $\phi$ is
holomorphic along the $V_{P_{j}}$: $\overline{\partial}\phi = 0$ also in the
directions transverse to $V_{P_{j}}$.
 \begin{proof}
We start with $(i)$\footnote{When $\Omega$ is smooth, $(i)$ is a
special case of Theorem 1.3 in \cite{ChengJinWang}.}. We want to estimate
$\|H^{q-1}_{\phi}f\|^{2}$ in such a way that we can use \eqref{compcond}. So fix
$\varepsilon>0$. In view of \eqref{Kohn}, we have
\begin{align}\label{practical}
 \|H^{q-1}_{\phi}f\|^{2} =
\langle\overline{\partial}^{*}N_{q}(\overline{\partial}\phi\wedge f),
\overline{\partial}^{*}N_{q}(\overline{\partial}\phi\wedge f)\rangle 
 = \langle N_{q}(\overline{\partial}\phi\wedge f), \overline{\partial}\phi\wedge
f\rangle 
\end{align}
for $f\in K^{2}_{(0,q-1)}(\Omega)$.
The (absolute value of the) inner product on the right hand side is 
\begin{align}\label{innerprod}
\sum_{j=1}^{n}\;\sideset{}{'}\sum_{|K|=q-1} \int_{\Omega}
(N_{q}(\overline{\partial}\phi\wedge f))_{jK} 
\overline{(\partial\phi/\partial\overline{z_{j}})f_{K}}dV 
\lesssim \sum_{j=1}^{n}\|(\partial\overline{\phi}/\partial
z_{j})N_{q}(\overline{\partial}\phi\wedge f)\|\,\|f\|,
\end{align}
with the usual notation $f=\sideset{}{'}\sum_{|K|=q-1}f_{K}d\overline{z_{K}}$,
the $'$ indicating summation over increasing multi--indices, and
$jK=(j,k_{1},\ldots,k_{q-1})$.  Using the inequality $2ab\leq (a^{2}/\varepsilon
+ \varepsilon b^{2})$, where $a,b>0$, gives
\begin{align}\label{eq13}
 2\|(\partial\overline{\phi}/\partial z_{j})N_{q}(\overline{\partial}\phi\wedge
f)\|\,\|f\|\leq \frac{1}{\varepsilon}\|(\partial\overline{\phi}/\partial
z_{j})N_{q}(\overline{\partial}\phi\wedge f)\|^{2} + \varepsilon \|f\|^{2}.
\end{align}

Because $(\partial\overline{\phi}/\partial z_{j})$ vanishes on
$\cup_{j=1}^{N}\overline{V_{P_{j}}}$, it is a compactness multiplier on
$(0,q)$--forms (\cite[Proposition 1, Theorem 3]{CelikStraube09}): for all
$\varepsilon^{\prime}>0$, there is a constant $C_{\varepsilon^{\prime},j}$ such
that
\begin{align}\label{compmultiplier}
 \|(\partial\overline{\phi}/\partial z_{j})u\|^{2} \leq
\varepsilon^{\prime}(\|\overline{\partial}u\|^{2} +
\|\overline{\partial}^{*}u\|^{2}) + C_{\varepsilon^{\prime},j}\|u\|_{-1}^{2},
\end{align}
$u\in dom(\overline{\partial})\cap dom(\overline{\partial}^{*})\subset
L^{2}_{(0,q)}(\Omega)$. 
We apply \eqref{compmultiplier} to the form
$u=N_{q}(\overline{\partial}\phi\wedge f)$ on the right hand side of
\eqref{eq13}, with $\varepsilon^{\prime}=\varepsilon^{2}$, to obtain that the
right hand side of \eqref{eq13} is dominated by 
\begin{align}\label{hankelest}
\varepsilon\|\overline{\partial}^{*}N_{q}(\overline{\partial}\phi\wedge f)\|^{2}
+ C_{\varepsilon,j}\|N_{q}(\overline{\partial}\phi\wedge f)\|_{-1}^{2} +
\varepsilon\|f\|^{2} \\
 \lesssim \varepsilon\|f\|^{2} +
C_{\varepsilon,j}\|N_{q}(\overline{\partial}\phi\wedge f)\|_{-1}^{2};
\end{align}
the constants involved (other than $C_{\varepsilon,j}$) are independent of
$\varepsilon$ and $f$ (but not of $\phi$). We have used here that
$N_{q}(\overline{\partial}\phi\wedge f)\in dom(\overline{\partial})\cap
dom(\overline{\partial}^{*})$ and that
$\overline{\partial}N_{q}(\overline{\partial}\phi\wedge f)=0$ (since
$\overline{\partial}f=0$). Combining \eqref{practical} --\eqref{eq13} and
\eqref{hankelest} and summing over $j=1,\ldots,n$ results in the estimate we are
looking for: 
\begin{align*}
 \|H^{q-1}_{\phi}f\|^{2} \lesssim \varepsilon\|f\|^{2} +
C_{\varepsilon}\|N_{q}(\overline{\partial}\phi\wedge f)\|_{-1}^{2}\;;\; 
f\in K^{2}_{(0,q-1)}(\Omega),
\end{align*}
with $C_{\varepsilon} = max\{C_{\varepsilon,j}\,|\,j=1,\ldots,n\}$.

The operator $f\rightarrow N_{q}(\overline{\partial}\phi\wedge f)$ is continuous
from $K^{2}_{(0,q-1)}(\Omega)$ to $L^{2}_{(0,q)}(\Omega)$, hence is compact to
$W^{-1}_{(0,q)}(\Omega)$ (see e.g.  (\cite[Theorem 1.4.3.2]{GrisvardBook}: 
${W^{1}_{0}}(\Omega)\hookrightarrow L^{2}(\Omega)$ is compact; by duality, 
so is $L^{2}(\Omega)\hookrightarrow W^{-1}(\Omega)$). Therefore
\eqref{compcond} is satisfied for $T=H^{q-1}_{\phi}$, with
$X=K^{2}_{(0,q-1)}(\Omega)$, $Y=L^{2}_{(0,q-1)}(\Omega)$,
$Z_{\varepsilon}=W^{-1}_{(0,q)}(\Omega)$, and
$S_{\varepsilon}=N_{q}(\overline{\partial}\phi\wedge f)\,:
\,K^{2}_{(0,q-1)}(\Omega)\rightarrow W^{-1}_{(0,q)}(\Omega)$, and
$H^{q-1}_{\phi}$ is compact. (In this particular case, $Z_{\varepsilon}$ and
$S_{\varepsilon}$ are independent of $\varepsilon$.)

$(ii)$ is an easy consequence of $(i)$. Any symbol as in $(ii)$ can be
approximated uniformly on $\overline{\Omega}$ by symbols in
$C^{1}(\overline{\Omega})$ and vanishing in a neighborhood of the compact set
$\cup_{j=1}^{N}\overline{V_{P_{j}}}$. These symbols satisfy the assumption in
$(i)$, and the corresponding (compact) Hankel operators converge to
$H^{q-1}_{\phi}$ in norm\footnote{Because $\|H^{q-1}_{\phi} -
H^{q-1}_{\phi_{j}}\| = \|H^{q-1}_{\phi - \phi_{j}}\| \leq \|\phi -
\phi_{j}\|_{L^{\infty}(\Omega)}$.}. Therefore, $H^{q-1}_{\phi}$ is compact as
well.\footnote{By \cite[Proposition 1, Theorem 3]{CelikStraube09}, $\phi$ is a
compactness multiplier for $(0,q)$--forms. When $\Omega$ is smooth (and $q=1$),
one can use \cite[Theorem 1]{CelikZeytuncu16}: symbols which are compactness
multipliers on a bounded smooth pseudoconvex domain produce compact Hankel
operators.}
\end{proof}

We are now ready to prove Theorem \ref{converse}.
\begin{proof}[Proof of Theorem \ref{converse}]
For the moment, fix $j$, $1\leq j\leq N$. $\overline{V_{P_{j}}}$ is a convex
subset of an affine subspace, and $\phi$ is holomorphic on $V_{P_{j}}$ and
continuous on the closure. Via dilatation, we can approximate $\phi$ uniformly
of $\overline{V_{P_{j}}}$, first by functions holomorphic in a relative
neighborhood of $\overline{V_{P_{j}}}$, and then (by trivial extension)
holomorphic in a neighborhood in $\mathbb{C}^{n}$ of $\overline{V_{P_{j}}}$.

To prove compactness of $H^{q-1}_{\phi}$, we use the sufficient condition in
\eqref{compcond}. So fix $\varepsilon>0$. For $j=1,\ldots,N$, choose open sets
$U_{j}$ in $\mathbb{C}^{n}$ that contain $\overline{V_{P_{j}}}$ and are pairwise
disjoint, and cutoff functions $\chi_{j}\in C^{\infty}_{0}(U_{j})$,
$0\leq\chi_{j}\leq 1$, which are identically equal to one in a neighborhood of
$\overline{V_{P_{j}}}$. In light of the previous paragraph, we can do that in
such a way that there exists a holomorphic function
$h_{j}$ on $U_{j}$ that first approximates $\phi$ to within $\varepsilon$ on
$\overline{V_{P_{j}}}$, and then by continuity to within $2\varepsilon$ on
$U_{j}\cap\overline{\Omega}$ (shrinking $U_{j}$ if necessary). Also set
$\chi_{0} =(1-\sum_{j=1}^{N}\chi_{j})\in C(\overline{\Omega})$. With this setup,
we split $H^{q-1}_{\phi}$ as follows:
\begin{align}\label{split}
 H^{q-1}_{\phi} = H^{q-1}_{(\chi_{0}\phi)} + H^{q-1}_{\sum_{j=1}^{N}\chi_{j}\phi} 
= H^{q-1}_{(\chi_{0}\phi)} + H^{q-1}_{\sum_{j=1}^{N}\chi_{j}h_{j}} 
+ H^{q-1}_{\sum_{j=1}^{N}\chi_{j}(\phi-h_{j})}\;.
\end{align}
$\chi_{0}\phi$ vanishes on $\cup_{j=1}^{N}\overline{V_{P_{j}}}$, and 
Lemma \ref{preplemma}, part $(ii)$, implies that $H^{q-1}_{(\chi_{0}\phi)}$ 
is compact. It remains to consider the remaining terms on the right hand 
side of \eqref{split}.

We have $\sum_{j=1}^{N}\chi_{j}h_{j}\in C^{\infty}(\overline{\Omega})$ and 
$\overline{\partial}(\sum_{j=1}^{N}\chi_{j}h_{j}) = \sum_{j=1}^{N}h_{j}\overline{\partial}\chi_{j}$. 
This symbol therefore satisfies the assumptions in part $(i)$ of Lemma \ref{preplemma}, 
and the corresponding Hankel operator is compact. Because 
$\|\chi_{j}(\phi - h_{j})\|_{L^{\infty}(\Omega)} \leq 2\varepsilon$, and the $\chi_{j}$ 
have disjoint supports, the norm of the last operator in \eqref{split} is 
bounded by $2\varepsilon$, and we conclude
\begin{align}\label{chi-k} 
 \|H^{q-1}_{\phi}f\| \leq 2\varepsilon\|f\| + \|H^{q-1}_{(\chi_{0}\phi + \sum_{j=1}^{N}\chi_{j}h_{j})}f\|\;.
\end{align}

Estimate \eqref{chi-k} is  \eqref{compcond} (after rescaling $\varepsilon$) with
$X=K^{2}_{(0,q-1}(\Omega)$, $Y=L^{2}_{(0,q-1}(\Omega)$,
$Z_{\varepsilon}=L^{2}_{(0,q-1)}(\Omega)$, and the compact operator
$S_{\varepsilon}=H^{q-1}_{(\chi_{0}\phi + \sum_{j=1}^{N}\chi_{j}h_{j})} 
= H^{q-1}_{\chi_{0}\phi} + H^{q-1}_{\sum_{j=1}^{N}\chi_{j}h_{j}}$ (in contrast to 
$Z_{\varepsilon}$, $S_{\varepsilon}$ does depend on $\varepsilon$, via $\chi_{j}$ 
and $h_{j}$). As  $\varepsilon > 0$ was arbitrary, we conclude that $H^{q-1}_{\phi}$ 
is compact and complete the proof of Theorem \ref{converse}.
\end{proof}

%%%%%%%%%%%%%%%%%%%%%%%%%%%%%%%%%%%

\end{document}